\documentclass[12pt]{article}

\usepackage{setspace}

\usepackage{amsthm, amssymb, amstext}
\usepackage[fleqn]{amsmath}
\usepackage{latexsym}
\usepackage[dvips]{graphicx}
\usepackage{comment}
\usepackage{hyperref}
\usepackage[capitalise,nameinlink]{cleveref}
\usepackage{mathtools}

\usepackage{todonotes}
\usepackage{comment}

\newtheorem{theorem}{Theorem}[section]
\newtheorem{lemma}[theorem]{Lemma}

\newtheorem{conjecture}[theorem]{Conjecture}

\newtheorem{question}[theorem]{Question}

\theoremstyle{definition}

\definecolor{myRed}{rgb}{0.68, 0.05, 0.0}
\colorlet{myBlue}{blue!70!black}
\colorlet{myViolet}{myBlue!55!myRed}
\definecolor{darkraspberry}{rgb}{0.53, 0.15, 0.34}
\definecolor{olive}{rgb}{0.42, 0.56, 0.14}

\hypersetup{
	colorlinks=true,
        linkcolor=myBlue, 
        citecolor=myBlue,
	bookmarksopen=true,
	bookmarksnumbered,
	bookmarksopenlevel=2,
	bookmarksdepth=3
}

\usepackage{authblk}

\DeclareMathOperator{\tw}{tw}
\DeclareMathOperator{\sn}{sn}
\DeclareMathOperator{\ta}{tree-\alpha}
        
\title{On treewidth and  maximum cliques}

\author[1]{Maria Chudnovsky}
\author[2]{Nicolas~Trotignon}

\affil[1]{Princeton University, Princeton, NJ 08544, USA}
\affil[2]{Univ Lyon, EnsL, CNRS, LIP, F-69342, Lyon Cedex 07, France}

\date{\today}

\begin{document}
\maketitle

\begin{abstract}
  We construct classes of graphs that are variants of the so-called
  \emph{layered wheel}.  One of their key properties is that while the
  treewidth is bounded by a function of the clique number, the
  construction can be adjusted to make the dependence grow arbitrarily.
  Some of these classes provide counter-examples to several
  conjectures.  In particular, the construction includes  hereditary classes of
  graphs whose treewidth is bounded by a function of the clique number
  while the tree-independence number is unbounded, thus disproving a
  conjecture of Dallard, Milani\v c and \v Storgel [Treewidth versus
  clique number. II. Tree-independence number. {\it Journal of
    Combinatorial Theory, Series B}, 164:404--442, 2024]. The
  construction can be further adjusted to provide, for any fixed
  integer $c$, graphs of arbitrarily large treewidth that contain no
  $K_c$-free graphs of high treewidth, thus disproving a conjecture of
  Hajebi [Chordal graphs, even-hole-free graphs and sparse
  obstructions to bounded treewidth, arXiv:2401.01299, 2024].
\end{abstract}

\section{Introduction}

Graphs in this paper are oriented and infinite (with neither loops nor
multiple edges).  However, this is only for technical reasons and most of
our results will be about finite and simple graphs.

A \emph{clique} in a graph is a set of pairwise adjacent vertices and
an \emph{independent set} is a set of pairwise non-adjacent vertices.
The maximum number of vertices in a clique (resp.\ independent set)
in a graph $G$ is denoted by $\omega(G)$ (resp.\ $\alpha(G)$).  We
denote by $\chi(G)$ the \emph{chromatic number of $G$}, that is the
minimum number of colors needed to color vertices of $G$ in such a way
that adjacent vertices receive different colors.  A \emph{hole} in a
graph is a chordless cycle of length at least~4.  It is \emph{even} if
it contains an even number of vertices.  We denote by $N(v)$ the
neighborhood of a vertex $v$ and set $N[v]= N(v) \cup \{v\}$.  A class of
graphs is \emph{hereditary} if it closed under taking induced
subgraphs.

A \emph{tree decomposition} of a graph $G$ is a pair
$\mathcal T = (T, (X_s)_{s\in V (T )})$ where $T$ is a tree and every
node $s \in T$ is assigned a set $X_s \subseteq V (G)$ called a
\emph{bag} such that the following conditions are satisfied: every
vertex is in at least one bag, for every edge $uv \in E(G)$ there
exists a bag $X_s$ such that $\{u, v \} \subseteq X_s$, and for every
vertex $u \in V (G)$, the set $\{s \in V (T) \mid u \in X_s\}$ induces a
connected subgraph of $T$.  The \emph{width} of $\mathcal T$ is the
maximum value of $|X_s| - 1$ over all $s \in V (T)$. The
\emph{independent width} of $\mathcal T$ is the maximum value of
$\alpha(X_s)$ over all $s \in V (T)$. The \emph{treewidth} of a graph
$G$, denoted by $\tw(G)$, is the minimum width of a tree decomposition
of $G$. The \emph{tree-independence number} of a graph $G$, denoted by
$\ta(G)$, is the minimum independent width of a tree decomposition of
$G$.  It was first defined by Yolov
in~\cite{doi:10.1137/1.9781611975031.16} and rediscovered
independently by Dallard, Milani\v c and \v Storgel
in~\cite{DBLP:journals/jctb/DallardMS24}. The treewidth, and more
recently the tree-independence number, attracted some attention, see
for instance the introduction
of~\cite{DBLP:journals/jctb/DallardMS24} or~\cite{DBLP:journals/jgt/AbrishamiACHSV24a}.

The main contribution of this paper is a variant of the so-called
\emph{layered wheel}.  It was first introduced by Sintiari and
Trotignon in~\cite{DBLP:journals/jgt/SintiariT21} to provide graphs of
arbitrarily large treewidth that exclude several kinds of induced
subgraphs, such as $K_4$ and even holes, or triangles and thetas (not
worth defining here).  Our variant may contain cliques of any
size. For every integer $\ell \geq 4$, we construct a variant that
contains only holes of length at least $\ell$.  Our construction
provides answers to questions and counter-examples to conjectures due
to different authors, all about the treewidth and the
tree-independence number in hereditary classes of graphs, as we
explain now.

\subsection*{Conjectures and questions}

A hereditary class of graphs $\cal C$ is said to be
$(\tw, \omega)$-bounded if there exists a function $g$ such that the
treewidth of any graph $G \in \cal C$ is at most $g(\omega(G))$. In
\cite{dallardMS:tw1} and \cite{DMS_JCTB2024}, Dallard, Milani\v c and \v Storgel
asked whether every $(tw, \omega)$-bounded class of graphs is in fact
\emph{polynomially $(tw, \omega)$-bounded}.  We rephrase this question
formally as follows.

\begin{question}[Dallard, Milani\v c and \v Storgel, see
{\cite[Question 8.4]{DMS_JCTB2024}}]
  \label{Q:allFunctions}
  For every $(\tw, \omega)$-bounded class of graphs $\cal C$, does
  there exist a polynomial $g$ such that every graph $G\in \cal C$
  satisfies $\tw(G) \leq g(\omega(G))$ ?
\end{question}

Our construction provides a negative answer to this question, see
\cref{th:allFunctions} below.  In~\cite[Lemma
3.2]{DBLP:journals/jctb/DallardMS24} it is observed that the answer is
affirmative for classes with bounded tree-independence number.  It is also
observed that hereditary classes of graphs with bounded
tree-independence number are $(\tw, \omega)$-bounded (this is an easy
consequence of Ramsey theorem, see~\cite{DMS_JCTB2024}).  The
following conjecture is proposed by Dallard, Milani\v c and \v
Storgel.

\begin{conjecture}[Dallard, Milani\v c and \v Storgel, see
  {\cite[Conjecture 8.5]{DMS_JCTB2024}}]
  \label{conj:dallardMiSto}
  Let $\cal C$ be a hereditary graph class. Then $G$ is
  $(\tw, \omega)$-bounded if and only if $\cal C$ has bounded
  tree-independence number.
\end{conjecture}

Our construction disproves this conjecture, even when the function
that bounds the treewidth is assumed to be a polynomial, see
\cref{th:dallardMiSto} below.

\medskip

Our construction also sheds light on certifying a large treewidth in a
graph $G$ by exhibiting some simpler substructure of $G$ of large
treewidth.  The well-known Grid Theorem by Robertson and
Seymour~\cite{DBLP:journals/jct/RobertsonS86} gives a neat certificate
when the substructure under consideration is a minor of $G$.  When the
substructure under consideration is an induced subgraph of $G$, the
situation is more complicated and is still the subject of much
research.  To understand this better, several questions and
conjectures (together with a survey) are proposed by Hajebi
in~\cite{hajebi:conj}.

\begin{conjecture}[Hajebi, see {\cite[Conjecture 1.14]{hajebi:conj}}]
  \label{conj:hajebiConj1}
  For every $t \geq 1$, every graph of large enough treewidth has an
  induced subgraph of treewidth $t$ which is either complete or
  $K_4$-free.
\end{conjecture}

Our construction (or more precisely a variant of it) disproves this
conjecture, even in a weaker form, where $K_4$ is replaced in the
statement by $K_c$ for any constant $c\geq 4$, see
\cref{th:hajebiConj} below.  Note that when $c\in \{1, 2\}$ the
statement is trivially false, and for $c=3$ it is already disproved
in~\cite{DBLP:journals/jgt/SintiariT21} with the construction of
$K_4$-free graphs of high treewidth with no even holes.  For every
integer $c$, a graph is \emph{$c$-degenerate} if each of its induced
subgraphs contains a vertex of degree at most~$c$.  A strengthening of
\cref{conj:hajebiConj1} is proposed in the same paper.

\begin{conjecture}[Hajebi, see {\cite[Conjecture 1.15]{hajebi:conj}}]
  \label{conj:hajebiConj2}
  For every integer $t \geq 1$, every graph of large enough treewidth
  has an induced subgraph of treewidth $t$ which is either complete,
  complete bipartite, or 2-degenerate.
\end{conjecture}

Our construction also disproves \cref{conj:hajebiConj2}, see
\cref{th:hajebiConj} below (even with ``2-degenerate'' replaced by
higher degeneracy).

\subsection*{Unanswered questions}

We would like to point out several questions that our construction
does not answer.  Our construction is easily seen to produce graphs whose treewidth
is logarithmic in the number of vertices.  This implies (see for
instance~\cite[Theorem 9.2]{abrishamiEtAl:twIII}) that the Max Weight
Independent Set problem is polynomial-time solvable for our
construction.  Hence, we believe the following question might still
have an affirmative answer, and that logarithmic treewidth might be a
key ingredient of an algorithm.

\begin{question}[Dallard, Milani\v c and \v Storgel, see~\cite{DBLP:journals/jctb/DallardMS24}]
Is the Max Weight Independent Set problem solvable in polynomial
time in every $(\tw, \omega)$-bounded graph class?
\end{question}

In our counter-example to \cref{conj:dallardMiSto}, we need that the class
is $(\tw, \omega)$-bounded by a super-linear function.  Hence, our
construction does not seem to help answering the following question.

\begin{question}
  Does every hereditary class of graph $\cal C$ such that for some
  constant $c$, every graph $G \in \cal C$ satifisfies
  $\tw(G) \leq c\,\omega(G)$, has bounded tree-independence number?
\end{question}

We leave as an open question the existence of an even-hole-free
variant of our construction.  If it exists, it might disprove the
following conjecture.

\begin{conjecture}[Hajebi, see~\cite{hajebi:conj}]
  For every $t \geq 1$, every even-hole-free graph of large enough
  treewidth has an induced subgraph of treewidth $t$ which is either
  complete or $K_4$-free.
\end{conjecture}

In our counter-example to \cref{conj:hajebiConj1}, we need a graph $G$
such that $\omega(G) = c+1$ to guarantee that $K_c$-free graphs in the
class have bounded treewidth. Hence, we do not know the answer to the
following question (quite suprisingly, \cref{th:allFunctions} does not
seem to help).

\begin{question}
  \label{q:hw}
  Does there exist a function $f: \mathbb N\setminus\{0\} \times \mathbb N\setminus\{0\}
  \rightarrow \mathbb N$ such that for all integers $k, t \geq
  1$, all graphs $G$ such that $\tw(G) \geq f(k, t)$ and $\omega(G) =
  k$ contain an induced subgraph $H$ such that $\tw(H) \geq t$ and
  $\omega(H) = k-1$? 
\end{question}

It might be of interest to study what induced subgraphs are
contained in our construction.  The original layered wheels
from~\cite{DBLP:journals/jgt/SintiariT21} suggest that maybe the
so-called 3-path-configurations (see the definition
in~\cite{DBLP:journals/jgt/SintiariT21})  are not contained in our
construction, or in some variant of our construction.

It is easy to check that our construction contains as an induced
subgraph an already known construction of Corneil and Rotics, designed
to find graphs of large cliquewidth (not worth defining here) and
small treewidth, see~\cite{DBLP:journals/siamcomp/CorneilR05}.  We
therefore wonder whether our construction improves bounds about
differents notions of widths of graphs. 

\subsection*{Tools to bound the treewidth}

To prove our results, we need to bound the treewidth and the
tree-independence number of graphs produced by our construction. For
that we rely on several known concepts and theorems.

To bound the treewidth and the tree-independence number from below, we
rely on \emph{clique minors}.  For a graph $G$, two disjoint sets
$Y\subseteq V(G)$ and $Z\subseteq V(G)$ are \emph{adjacent} if some
edge of $G$ has an end in $Y$ and an end in $Z$.  A \emph{clique
  minor} of a graph $G$ is a family of disjoint, connected, and
pairwise adjacent subsets $L_1$, \dots, $L_t$ of $V(G)$.  The
following is a well-known consequence of the Helly property of
subtrees of trees.

\begin{lemma}
  \label{l:Kminor}
  If $\mathcal T = (T, (X_s)_{s\in V (T )})$ is a tree decomposition
  of some graph $G$ and $(L_1, \dots, L_t)$ is a clique minor of $G$,
  then there exists $s\in V(T)$ such that $X_s$ contains at least one
  vertex of each $L_i$, $i\in \{1, \dots, t\}$. In particular,
  $\tw(G) \geq t-1$ and the independent width of $\mathcal T$ is at
  least $\alpha(G[X_s])$.
\end{lemma}

To bound the treewidth from above, we rely on \emph{balanced
  separations}. A \emph{separation} of a graph $G$ is a pair $(A, B)$
of subsets of $V(G)$ such that $A \cup B = V(G)$ and no edge of $G$
has one end in $A \setminus B$ and the other in $B \setminus A$.  The
\emph{order} of the separation is $|A \cap B|$.  It is \emph{balanced}
if $|A \setminus B| \leq 2n / 3$ and $|B \setminus A| \leq 2n / 3$
where $n= |V(G)|$.  The \emph{separation number} $\sn(G)$ of $G$ is
the smallest integer $s$ such that every subgraph (or equivalently
induced subgraph) of $G$ has a balanced separation of order $s$.

\begin{theorem}[Dvor{\'{a}}k and Norin,
  see~\cite{DBLP:journals/jct/DvorakN19}]
  \label{th:DvNo}
  The treewidth of any graph $G$ is at most $15\sn(G)$. 
\end{theorem}

We also need the following classical results. A graph is
\emph{chordal} if it contains no hole.  A classical characterization
due to Rose~\cite{rose:triangulated} tells that a graph $G$ is chordal
 if and only if all induced subgraphs $H$ of $G$ contain a
simplicial vertex (in~$H$), where a vertex is \emph{simplicial} if its
neighborhood is a clique (for our purpose, we only need this
characterization of chordal graphs).  The following is usually refered
to as the perfection of chordal graphs. 

\begin{theorem}[Dirac, see \cite{dirac:chordal}]
  \label{th:ch}
  If $G$ is a chordal graph, then $\chi(G) = \omega(G)$. 
\end{theorem}

\subsection*{Outline of the paper}

We define what we call the $(f, \ell)$-layered wheels in
\cref{sec:def}.  We study some of their structural properties in
\cref{sec:struct}.  The $(f, \ell)$-layered wheels are infinite graphs
and their finite induced subgraphs are studied in \cref{sec:finite}.
We provide the answers to questions and counter-examples to
conjectures in \cref{sec:app}.

\section{Definition of layered wheels}
\label{sec:def}

A function $f:\mathbb N\setminus\{0\} \rightarrow \mathbb N\setminus\{0\}$ is \emph{slow} if
$f(1)=1$, $f(2)=2$, $f(3)=3$ and for every $i\in \mathbb N\setminus\{0\}$,
$f(i) \leq f(i+1) \leq f(i) +1$.  Observe that a slow function is
non-decreasing.  Hence, every slow function is either ultimately constant
or tends to $+\infty$.

For every slow function $f$ and every integer $\ell \geq 4$, we define
an oriented graph $G$ called the \emph{$(f, \ell)$-layered wheel}.
The vertex-set of $G$ is countably infinite.  The set of arcs of $G$
is denoted by $A(G)$.  Note that all the theorems in this paper are
about the underlying graph of $G$.  The orientations of the arcs are
only used to facilitate the description of several subsets of
$V(G)$. Before defining $G$ precisely, we list five rules giving some
properties and terminology that help stating the formal description.

\begin{figure}[t]
  \begin{center}
    \includegraphics[width=11cm]{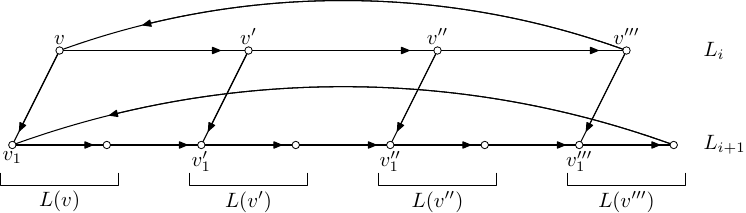}
  \end{center}
  \caption{Rules \ref{a:hole}--\ref{aN}\label{f:A}}
\end{figure}

\begin{enumerate}
\item\label{a:f} $V(G)$ is partitioned into sets
  $L_i$, ${i \in \mathbb N\setminus\{0\}}$,  called the \emph{layers} of
  $G$.

\item\label{a:hole} Each layer $L_i$, $i \in \mathbb N\setminus\{0\}$, induces a
  directed cycle of length at least $\ell$.

\item\label{a:orient} If $u$ and $v$ are adjacent vertices in different layers, say
  $v\in L_i$, $u\in L_j$ and $i<j$, then the arc linking them is
  oriented from $v$ to $u$.
  
\item\label{aN} For every $i\in \mathbb N\setminus\{0\}$ and every vertex
  $v\in L_i$, there exists a positive integer $n_v$ and a directed path
  $L(v) = v_1 \dots v_{n_v}$ such that $V(L(v)) \subseteq L_{i+1}$ and
  $L(v)$ contains all the neighbors of $v$ in $L_{i+1}$.  The
  vertex $v_1$ is adjacent to $v$. Moreover the paths $L(v)$,
  $v\in L_i$, are vertex-disjoint, $L_{i+1} = \bigcup_{v\in L_i} V(L(v))$
  and if $vv' \in A(G)$, then $v_{n_v}v'_1 \in A(G)$.

  It follows that a vertex $u \in L_{i+1}$ has at most one neighbor in
  $L_i$.  If such a neighbor $v$ exists, we say that $v$ is the
  \emph{parent} of $u$ and $u$ is a \emph{child} of~$v$. Observe that
  every vertex in $G$ has at least one child and at most one parent.

  Informally, the directed paths $L(v)$, $v\in L_i$, vertex-wise
  partition $L_{i+1}$, and parents in $L_i$ and their children in
  $L_{i+1}$ appear in the same cyclic order, see \cref{f:A}.

\item\label{a:l} If $v\in L_{i}$, then $v$ has at most $f(i)-1$
  neighbors in $\bigcup_{1\leq j <i} L_j$.  Moreover, these neighbors
  induce a (possibly empty) clique that contains at most one vertex in
  each layer.  We denote this clique by $N^{\uparrow}(v)$ and we set
  $N^{\uparrow}[v] = \{v\} \cup N^{\uparrow}(v)$.
\end{enumerate}

Now we define the $(f, \ell)$-layered wheel $G$.  The first layer
$L_1$ is a directed cycle of length $\ell$.  We suppose inductively
that for some integer $i\geq 1$, the graph induced by the $i$ first
layers (so $G[L_1 \cup \dots \cup L_{i}]$) is defined, and we explain
how to add the next layer $L_{i+1}$.  We assume that rules~\ref{a:f},
\dots, \ref{a:l} hold for layers $L_1$, \dots, $L_i$.  Since $f$ is
slow, $f(i) \in \{f(i+1)-1, f(i+1)\}$.  Hence, by rule~\ref{a:l}, for
every $v\in L_i$, $|N^\uparrow(v)| \leq f(i+1)-1$.

To fully describe $G[L_1 \cup \dots \cup L_{i+1}]$, it is sufficient
to define for each $v\in L_i$ the integer $n_v$, and for each vertex
$v_j\in L(v)$, the set $N^\uparrow(v_j)$; then $L_{i+1}$ can be
described as the vertex-set of the directed cycle formed by the
consecutive directed paths $L(v)$, so the neighborhood of every vertex
from $L_{i+1}$ is fully described.  So, let $v$ be a vertex in
$L_i$. There are two cases (see \cref{f:B1} and \cref{f:B2}):
  
\begin{figure}[p]
  \begin{center}
    \includegraphics[width=14cm]{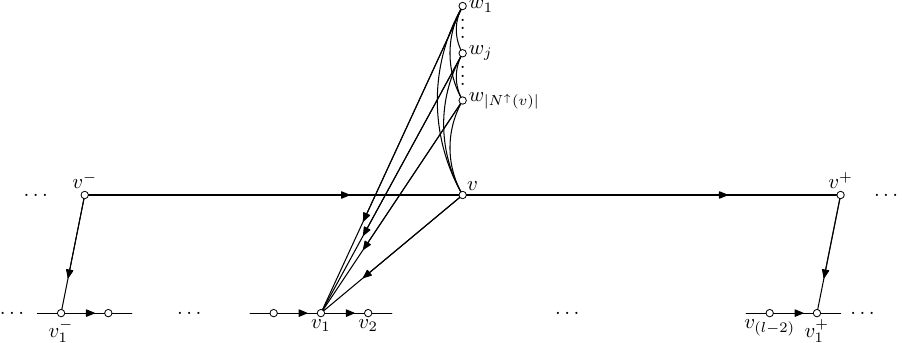}
  \end{center}
  \caption{Rule \ref{bS}\label{f:B1}}
\end{figure}

  \begin{figure}[p]
    \begin{center}
      \includegraphics[width=14cm]{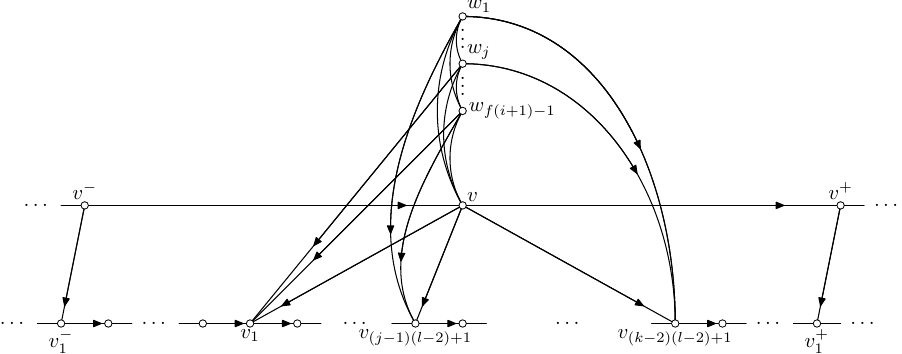}
    \end{center}
    \caption{Rule \ref{bB}\label{f:B2}}
  \end{figure}

\begin{enumerate}
\setcounter{enumi}{5}
\item\label{bS} If $|N^\uparrow(v)| < f(i+1)-1$, then we set $n_v= \ell-2$, so 
  $L(v) = v_1 \dots v_{\ell-2}$.   Moreover, we set:

  \begin{itemize}
  \item $N^\uparrow (v_1) = N^\uparrow[v]$
  \item
    $N^\uparrow (v_2) = \cdots = N^\uparrow (v_{\ell - 2}) =
    \emptyset$.
  \end{itemize}

\item\label{bB} If $|N^\uparrow(v)| = f(i+1)-1$, then by
  rule~\ref{a:l}, we have
  $N^\uparrow(v) = \{w_1, \dots, w_{f(i+1)-1}\}$, where for all
  $j\in \{1, \dots, f(i+1)-1\}$, $w_j\in L_{i_j}$, and
  $1\leq i_1 < i_2 < \dots < i_{f(i+1)-1} < i$.  We set
  $n_v = (f(i+1)-1) (\ell -2)$, so
  $L(v) = v_1 \dots v_{(f(i+1)-1)(\ell -2)}$. Moreover, we set:

  \begin{itemize}
  \item For all $j\in \{1, \dots, f(i+1)-1\}$,
    $$N^\uparrow(v_{(j-1)(\ell - 2) +1}) = N^\uparrow[v] \setminus
    \{w_j\}.$$
  \item For all $j\in \{1, \dots, f(i+1)-1\}$ and
    $j' \in \{2, \dots, \ell-2\}$,
    $$N^\uparrow(v_{(j-1)(\ell - 2) +j'}) = \emptyset.$$ 
  \end{itemize}

\end{enumerate}

The description of $G[L_1 \cup \dots \cup L_{i+1}]$ is now completed
(see \cref{f:ex} for an example).  Note that the new layer $L_{i+1}$
satisfies rules~\ref{a:f}--\ref{a:l}, so the inductive process can go on, and
$G$ is inductively defined.  We now give several informal statements
that will hopefully help the reader.

\begin{figure}[h]
  \begin{center} \includegraphics[width=13cm]{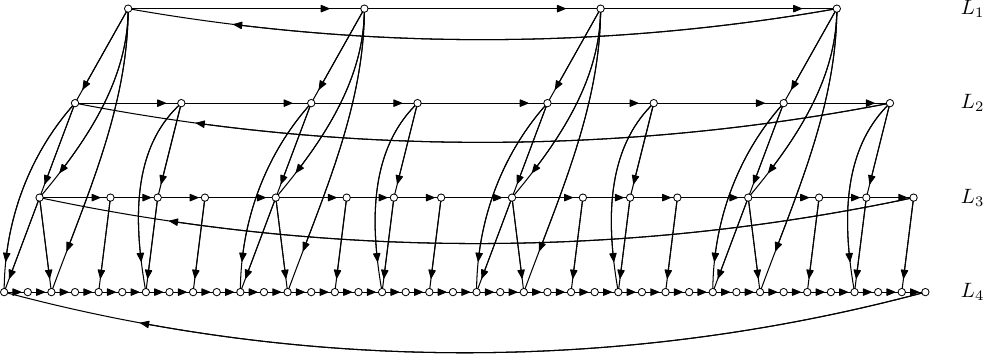}
  \end{center}
  \caption{Layers $L_1$ to $L_4$ of the $(f, 4)$-layered wheel when
    $f(4) = 3$ \label{f:ex}}
\end{figure}

\medskip

The $(f, \ell)$-layered wheel is infinite, but we are only interested
in its finite induced subgraphs, that form a hereditary class of
graphs.  Recall that the orientations of the arcs are here only to
help describing some sets of vertices later in the proofs. The
statements of the theorems are only about the underlying undirected
graph of $G$.

The layers of $G$ should be thought of as the sets of a clique minor.  More
precisely, it will be shown in \cref{l:layersCM} that for every pair
of layers $L_i, L_j$ there exists an edge with one end in $L_i$ and
the other in $L_j$.  It follows that for all integers $t\geq 1$, the
layers $L_1, \dots, L_t$ form a clique minor of
$G[L_1 \cup \dots \cup L_t]$, that has therefore treewidth at
least~$t-1$ by \cref{l:Kminor}.

The integer $\ell$ should be thought of as the length of a smallest
hole in $G$.  It controls the number of new vertices with no parent
that are introduced in each new layer.  The larger is $\ell$, the more of
them are needed to prevent creating short holes.

For every integer $i\geq 1$, the integer $f(i+1)$ should be thought of
as the size of the cliques that are allowed to be introduced when the
layer $L_{i+1}$ is built.  When a vertex $v\in L_i$ is such that
$|N^\uparrow(v)| < f(i+1)-1$, its children (all of which are in
$L_{i+1}$) can be made complete to $N^\uparrow[v]$ whithout creating a
clique of size larger than $f(i+1)$, and this explains the rule
\ref{bS}. When $|N^\uparrow(v)| = f(i+1)-1$, the children of $v$
cannot be made complete to $N^\uparrow[v]$ whithout creating a clique
of size larger than $f(i+1)$.  This explains why in rule \ref{bB}, in
the neighborhood of a child of $v$, we have to exclude one vertex from
$N^\uparrow[v]$.  Still, we want to give a chance of later
augmentation to as many cliques as possible, and this is why we create
$f(i+1)-1$ children of $v$ that cover all possible ways of extending a
clique of size $f(i+1)-1$ that contains~$v$.

The function $f$ should be thought of as the speed at which larger
cliques are introduced in the construction. We keep the flexibility of
tuning $f$ for different purposes. In particular it is convenient that
$f$ can be eventually constant.  Then the layered wheel will have
bounded clique size, and yet unbounded treewidth, see
\cref{th:hajebiConj}.  On the other hand, if $f$ tends to $+\infty$,
then the layered wheel provides a class of graphs whose treewidth is
bounded by a function of $\omega$ that is closely related to $f$, see
\cref{th:allFunctions} and \cref{th:dallardMiSto}.

\section{Structure of layered wheels}
\label{sec:struct}

Throughout this section, $f$ is a slow function, $\ell\geq 4$ is an
integer and $G$ is the $(f, \ell)$-layered wheel.

\begin{lemma}
  \label{l:lengthHoles}
  All holes of $G$ have length at least $\ell$. 
\end{lemma}

\begin{proof}
  Let $H$ be a hole of $G$ and $i$ be the maximum integer such that
  $H$ contains vertices of $L_i$.  By rule~\ref{a:hole}, if
  $V(H) = L_i$, then $H$ has length at least $\ell$, so we may assume
  that $V(H) \neq L_i$, implying that some vertex of $L_i$ is not in
  $H$ since $L_i$ induces a hole.  Let $u_1\dots u_j$ be a subpath of
  $L_i$ that is included in $H$ and maximal with respect to this
  property. We have $j>1$ since $N^\uparrow(u_1)$ is a clique by
  rule~\ref{a:l}.  Hence, by rule \ref{bS} or \ref{bB},
  $j\geq \ell-1$.  Hence, $H$ has length at least $\ell$.
\end{proof}

The following shows that the layers of $G$ form a clique minor of $G$.

\begin{lemma}
  \label{l:layersCM}
  For all integers $i\geq 1$ and $i'>i$, every vertex $u \in L_i$ has at
  least one neighbor in $L_{i'}$.   
\end{lemma}

\begin{proof}
  We prove the property by induction on $i'$.  If $i'=i+1$, then the
  conclusion follows directly from rule~\ref{aN}.

  Assume that the property holds for some fixed $i'\geq i+1$.  Let us
  prove it for $i'+1$.  Note that $i'+1\geq 3$. So since $f$ is slow,
  $f(i'+1) \geq 3$.  Let $w$ be a vertex in $L_i$.  By the induction
  hypothesis, $w$ has a neighbor $v\in L_{i'}$.  It is enough to check
  that the path $L(v)$ defined in~\ref{bS} or~\ref{bB} contains a
  neighbor of $w$.  We use the notation of rules~\ref{bS}
  and~\ref{bB}.

  If $|N^\uparrow(v)| < f(i'+1)-1$, then
  rule~\ref{bS} applies. Hence, $v_1$ is adjacent to $w$ since
  $N^\uparrow (v_1) = N^\uparrow[v]$.
  If $|N^\uparrow(v)| = f(i'+1)-1$, then rule~\ref{bB} applies.  Since
  $f(i'+1)-1\geq 2$, there exists $j\in \{1, \dots, f(i'+1)-1\}$ such that
  $w_j \neq w$. Hence, the vertex $v_{(j-1)(\ell - 2) +1}$ from $L(v)$ is
  adjacent to $w$ since
  $N^\uparrow(v_{(j-1)(\ell - 2) +1}) = N^\uparrow[v] \setminus
  \{w_j\}$.
\end{proof}

\begin{lemma}
  \label{l:omega}
  For all integers $i\geq 2$,
  $\omega(G[L_1 \cup \dots \cup L_i]) = f(i)$.
\end{lemma}

\begin{proof}
  Let us prove by induction on $i$ that for all $i\geq 1$, there
  exists a clique $K$ of $G[L_1 \cup \dots \cup L_i]$ on $f(i)$
  vertices such that $|K\cap L_i| = 1$.  This is clearly true for
  $i=1$.  We suppose inductively that such $K$ exists for some fixed
  $i\geq 1$ and call $v$ the unique vertex of $K\cap L_i$.  Since $f$
  is slow, $f(i+1) =f(i) +1$ or $f(i+1) = f(i)$.  In the former case,
  $|N^\uparrow(v)| = f(i)-1 = f(i+1)-2$, so rule \ref{bS} applies and
  $v$ has a unique child $u$ satisfying
  $N^\uparrow(u) = N^\uparrow[v]$.  Hence, $K\cup \{u\}$ is a clique on
  $f(i)+1 = f(i+1)$ vertices that contains $u$.  In the latter case,
  $|N^\uparrow(v)| = f(i)-1 = f(i+1)-1$, so rule \ref{bB} applies, and
  for any child of $u$ of $v$,
  $N^\uparrow(u) = N^\uparrow[v]\setminus \{w\}$ for some
  $w\in N^\uparrow(v)$.  Hence, we again find a clique satisfying the
  conclusion.

  We proved that $\omega(G[L_1 \cup \dots \cup L_i]) \geq f(i)$.  Let
  us prove the converse inequality. For $i=2$, it trivially holds.
  Hence, suppose $i\geq 3$.  Let $K$ be a maximum clique of
  $G[L_1 \cup \dots \cup L_i]$ and $j$ be the maximum integer such
  that $K\cap L_j \neq \emptyset$.  Note that $|K|\geq 3$.  By rules
  \ref{bS} and \ref{bB}, no two adjacent vertices in $L_j$ have a common
  neighbor in $L_{j'}$ if $j\geq j'$, so $|K\cap L_j| = \{u\}$ for
  some $u\in L_j$; consequently $K\subseteq N^\uparrow[u]$.  It
  follows by rule~\ref{a:l} that $|K|\leq f(i)$.
\end{proof}

An infinite directed path $P$ of $G$ is a \emph{vertical path starting
  in layer $i$} if $i\in \mathbb N \setminus \{0\}$,
$P= p_i p_{i+1} p_{i+2} \dots $ and for all $j\geq i$, $p_j\in L_j$.
Observe that $P$ may not be induced.

\begin{lemma}
  \label{l:pathD}
  Let $P= p_i p_{i+1} p_{i+2} \dots $ and  $Q=q_i q_{i+1} q_{i+2}
  \dots $ be two vertical paths starting in the same layer. If
  $p_i\neq q_i$, then $V(P) \cap V(Q) = \emptyset$.  
\end{lemma}

\begin{proof}
  Otherwise, the  common vertex of $P$ and $Q$ in $L_j$ such that $j$
  is minimal has two parents, a contradiction to rule~\ref{aN}.  
\end{proof}

\begin{lemma}
  \label{l:pathH}
  If $p_i p_{i+1} p_{i+2} \dots $ is a vertical path, then for all
  $j\geq i$, $$N^\uparrow[p_j] \subseteq V(P) \cup N^\uparrow(p_i).$$ 
\end{lemma}

\begin{proof}
  Let us prove the lemma by induction on $j$.  If $j=i$ it trivially
  holds, so suppose $j>i$.  By rules \ref{bS} or \ref{bB},
  $N^\uparrow(p_j) \subseteq N^\uparrow[p_{j-1}]$.  Hence, by the
  induction hypothesis,
  $N^\uparrow[p_j] \subseteq V(P) \cup N^\uparrow(p_i)$.
\end{proof}

Let $p$ and $q$ be two vertices of $L_i$. We denote by
$p \overrightarrow{L_i} q$ the vertex-set of the unique directed path
of $G[L_i]$ from $p$ to $q$.  Note that if $p=q$, then
$p \overrightarrow{L_i} q = \{p\}$ and if $qp\in A(G)$, then
$p \overrightarrow{L_i} q = L_i$.  We set
$p \overleftarrow{L_i} q = \{p, q\} \cup (L_i \setminus p
\overrightarrow{L_i} q)$.  Observe that $G[p \overrightarrow{L_i} q]$
and $G[p \overleftarrow{L_i} q]$ edge-wise partition $G[L_i]$, so
$(p \overrightarrow{L_i} q, p \overleftarrow{L_i} q)$ is a separation
of $G[L_i]$ of order~2 (or 1 if $p=q$) since
$p \overrightarrow{L_i} q \cap p \overleftarrow{L_i} q = \{p, q\}$.

Let $P = p_ip_{i+1}\dots$ and $Q = q_iq_{i+1}\dots$ be two vertical
paths starting in the same layer $L_i$.  We define
$$A(P, Q) = \left(\bigcup_{u \in p_i \overrightarrow{L_i} q_i} N^\uparrow[u]\right) \cup \bigcup_{j>
  i} p_j \overrightarrow{L_j} q_j$$ and
$$B(P, Q) = \left(\bigcup_{1 \leq j \leq i} L_j\right) \cup \bigcup_{j> i} p_j \overleftarrow{L_j} q_j.$$

\begin{lemma}
  \label{l:apb}
  If $P = p_ip_{i+1}\dots$ and $Q= q_iq_{i+1}\dots$ are two vertical paths starting in the same layer
  $L_i$, then:
 $$A(P, Q) \cup B(P, Q) = V(G)$$

 and
 
  $$A(P, Q) \cap B(P, Q) = V(P) \cup V(Q) \cup \bigcup_{u \in V(p_i \overrightarrow{L_i} q_i)}
  N^\uparrow[u].$$
  \end{lemma}

\begin{proof}
  By its definition, $B(P, Q)$ contains all layers $L_1, \dots, L_i$.
  Vertices of some layer $L_j$, $j>i$,  are all either in $p_j
  \overrightarrow{L_j} q_j$ or $p_j \overleftarrow{L_j} q_j$ since
  these two sets form a separation of $L_j$. Hence, they are either in
  $A(P, Q)$ or in $B(P, Q)$.  This proves the first equality.

  The only vertices of $A(P, Q)$ that are in layers $L_1, \dots,
  L_i$ are those from
  $\bigcup_{u \in p_i \overrightarrow{L_i} q_i} N^\uparrow[u]$, and it
  turns out that they are all in $B(P, Q)$.  In the next layers (so
  the $L_j$'s, $j>i$), the only vertices that are both in $A(P, Q)$
  and $B(P, Q)$ are the vertices of $P$ and $Q$ since for all $j>i$,
  $p_j \overrightarrow{L_j} q_j \cap p_j \overleftarrow{L_j} q_j =
  \{p_j, q_j\}$.  This proves the second equality.
\end{proof}

A vertex $w\in L_j$ is an \emph{ancestor} of a vertex $u\in L_i$ if
$wu\in A(G)$ and $j<i$.  A vertex $w\in L_j$ is a \emph{descendant}
of a vertex $u\in L_i$ if $uw\in A(G)$ and $j>i$.

\begin{lemma}
  \label{l:hb}
  If $P = p_ip_{i+1}\dots$ and $Q= q_iq_{i+1}\dots$ are two vertical paths starting in the same layer
  $L_i$ and $u\in A(P, Q) \setminus B(P, Q)$, then all the descendants of
  $u$ are in $A(P, Q) \setminus B(P, Q)$ and all ancestors of $u$ are
  in $A(P, Q)$.
\end{lemma}

\begin{proof}
  Let us first prove the claim about the descendants of $u$.  Suppose
  that the claim does not hold.  Then there exists $uw\in A(G)$ such
  that $u\in (A(P, Q) \setminus B(P, Q)) \cap L_j$,
  $w\in B(P, Q) \cap L_{j'}$ and $j'>j$. We choose such a pair $u, w$
  subject to the minimality of $j'-j$.  Since $uw\in A(G)$, by rule
  \ref{bS} or \ref{bB}, $w$ has a parent $v$ such that $u\in N[v]$.
  We have $v\neq u$ because otherwise, $w$ is a child of $u$ and by
  rule~\ref{aN}, the children of $u$ are in the interior of
  $p_{j+1} \overrightarrow{L_{j+1}}q_{j+1}$, so in
  $A(P, Q) \setminus B(P, Q)$.  If $v\in V(P)$, then by
  \cref{l:pathH}, $u\in V(P) \cup N^\uparrow(p_i)$, a contradiction to
  $u\in A(P, Q) \setminus B(P, Q)$.  So $v\notin V(P)$, and
  symmetrically $v\notin V(Q)$.  Hence, $v$ and $w$ (if
  $v\in p_{j'-1} \overrightarrow{L_{j'-1}}q_{j'-1}$) or $u, v$ (if
  $v\in p_{j'-1} \overleftarrow{L_{j'-1}}q_{j'-1}$) contradicts the
  minimality of $j'-j$.

  Let us now prove the claim about the ancestors of $u$.  Suppose that
  the claim does not hold.  Then there exists $wu\in A(G)$ such that
  $u\in (A(P, Q) \setminus B(P, Q)) \cap L_j$,
  $w\in (B(P, Q) \setminus A(P, Q))\cap L_{j'}$ and $j'<j$. We choose
  such a pair $u, w$ subject to the minimality of $j-j'$.  Since
  $wu\in A(G)$, by rule \ref{bS} or \ref{bB}, $u$ has a parent $v$
  such that $w\in N[v]$.  If $v\in V(P)$, then by \cref{l:pathH},
  $w\in V(P) \cup N^\uparrow(p_i)$, a contradiction to
  $w\in B(P, Q) \setminus A(P, Q)$.  So $v\notin V(P)$, and
  symmetrically $v\notin V(Q)$.  If $v\in p_i\overrightarrow{L_i}q_i$,
  then $w\in N^\uparrow[v]$, a contradiction to
  $w\in B(P, Q) \setminus A(P, Q)$.  Hence, by rule~\ref{aN}, $j\geq i+2$
  and
  $v\in p_{j-1}\overrightarrow{L_{j-1}} q_{j-1} \setminus \{p_{j-1},
  q_{j-1}\}$.  It follows that
  $v\in (A(P, Q) \setminus B(P, Q)) \cap L_{j-1}$.  Hence, $w$ and $v$
  contradict the minimality of $j-j'$ unless $j-j'=1$, in which case
  $w$ and $u$ contradict rule~\ref{aN}.
\end{proof}

\begin{lemma}
  \label{l:sep}
  If $P = p_ip_{i+1}\dots$ and $Q= q_iq_{i+1}\dots$ are two vertical
  paths starting in the same layer $L_i$, then $S= (A(P, Q), B(P, Q))$
  is a separation of~$G$.
\end{lemma}

\begin{proof}
  Suppose that $S$ is not a separation. Since by \cref{l:hb}
  $V(G) = A(P, Q) \cup B(P, Q)$, there exists in $G$ an edge $uv$ such
  that $u\in A(P, Q) \setminus B(P, Q)$ and
  $v\in B(P, Q) \setminus A(P, Q)$.  If $u$ and $v$ are in the same
  layer $L_j$, then $j>i$ because $B(P, Q)$ contains all layers $L_1$,
  \dots, $L_i$; so we have
  $u\in p_j\overrightarrow{L_j}q_j \setminus \{p_j, q_j\}$ and
  $v\in p_j\overleftarrow{L_j}q_j \setminus \{p_j, q_j\}$, a
  contradiction since $u$ and $v$ are adjacent.  Otherwise, $v$ is a
  descendant or an ancestor of $u$, so by \cref{l:hb}, $v\in A(P, Q)$,
  a contradiction again.
\end{proof}

\section{Finite induced subgraphs of layered wheels}
\label{sec:finite}

Throughout this section, $f$ is a slow function, $\ell\geq 4$ is
an integer and $G$ is the $(f, \ell)$-layered wheel.  Moreover, we
consider a finite set $X\subseteq V(G)$ and an integer $k\geq 1$ such
that $\omega(G[X]) \leq k$ and we study $G[X]$.  We set $n=|X|$.

\begin{lemma}
  \label{l:chordal}
  If $X$ contains at most one vertex in each layer of $G$, then $G[X]$
  is a chordal graph.
\end{lemma}

\begin{proof}
  Consider the maximum integer $i$ such that $L_i\cap X \neq \emptyset$.
  The unique vertex $v$ of $L_i\cap X$ has all its
  neighbors are in layers  $L_j$ such that $j<i$ (it has no neighbor in
  $L_i$ by assumption, and no neighors in $L_j$, $j>i$, by the maximality
  of $i$).   So $N(v) \cap X = N^\uparrow(v) \cap X$ and $v$ is
  simplicial by rule~\ref{a:l}.  This proof can be applied to any induced
  subgraph of $G[X]$, so $G[X]$ is chordal.  
\end{proof}

The arc $vu \in A(G)$ is \emph{augmenting with respect to $X$} if $u$
is a child of $v$ and
$$N^{\uparrow}(u) \cap X = N^{\uparrow}[v] \cap X.$$  In what
follows, we will omit to write ``with respect to $X$'' since $X$ is
fixed for the entire section.

\begin{lemma}
  \label{l:aug}
  If $v \in L_i$ and $f(i+1)\geq k+2$, then there exists at least one
  child $u$ of $v$ such that $vu$ is augmenting.
\end{lemma}

\begin{proof}
  Assume first that $|N^{\uparrow}(v)| < f(i+1) - 1$.  Then by rule
  \ref{bS}, the edge $vu$ is augmenting, where $u$ is the only child
  of $v$.  Indeed, we have $N^{\uparrow}(u) = N^{\uparrow}[v]$, so
  $N^{\uparrow}(u)\cap X = N^{\uparrow}[v] \cap X$ trivially holds.
  Thus we may assume that $|N^{\uparrow}(v)| \geq f(i+1) - 1$. Now:
  $$\begin{array}{rl} 
    f(i)-1&\geq |N^{\uparrow}(v)| \text{ by rule \ref{a:l}} \\ 
          &\geq f(i+1)-1 \text{ by assumption} \\
          &\geq f(i)-1 \text{ because $f$ is slow}
  \end{array}$$

  So $|N^{\uparrow}(v)| = f(i+1)-1 \geq k+1$. Since
  $\omega(G[X]) \leq k$, there exists $w\in N^{\uparrow}(v)\setminus X$.
  By rule \ref{bB}, $v$ has a child $u$ such that
  $N^{\uparrow}(u) = N^{\uparrow}[v]\setminus \{w\}$. Since $w\notin X$,
  we have $N^{\uparrow}(u) \cap X  = N^{\uparrow}[v] \cap X$, so the
  edge $vu$ is augmenting.
\end{proof}

For every vertex $v$ in $G$, we denote by $a(v)$ the \emph{augmenting
  child of $v$} that is a vertex defined as follows: if $v$ has a
child $u$ such that $vu$ is augmenting, then we choose such a child
$u$, and set $a(v) = u$. Otherwise, we choose any child $u$ of $v$ and
set $a(v) = u$.  There might be many ways to choose $a(v)$, but we
choose one of them and keep it for the rest of the proof.  For every
vertex~$v$, there exists a vertical path induced by $\{v,  a(v),
a(a(v)),  \dots\}$ that
we call the \emph{augmenting path out of $v$}.

When $f$ is a slow function, we set
$F(k) = \sup\{i \in \mathbb N\setminus\{0\} \mid f(i) \leq k\}$.  The function $F$
should be thought of as the maximum number of layers where $f$ is
at most $k$.  Observe that in the case where $f$ is eventually
constant, say $f(i) = c$ for all sufficiently large $i$, we have
$F(k) = +\infty$ for all $k\geq c$.

\begin{lemma}
  \label{l:sizeAug}
  If $v \in V(G)$ and $P$ is the augmenting path out of $v$, then
  $(V(P) \cap (\bigcup_{i\geq F(k+1)} L_i)) \cap X$ is a clique.  In
  particular, $$|V(P) \cap X| \leq F(k+1) + k -1.$$
\end{lemma}

\begin{proof}
  If $F(k+1) = +\infty$, then the conclusion trivially holds (in particular
  $\{i \in \mathbb N\setminus\{0\} \mid i \geq F(k+1)\} = \emptyset$).  Otherwise,
  $V(P) \cap (\bigcup_{i\geq F(k+1)} L_i)$ induces an infinite  vertical path
  $p_1p_2\dots $,  and each $p_j$ for $j\geq 1$ is in a layer
  $L_i$ such that $f(i+1) \geq k+2$.  Hence, by \cref{l:aug}, there
  exists an augmenting arc $p_ju$ for all $j\in \mathbb N\setminus\{0\}$, so
  by the definition of  augmenting paths, the arc $p_jp_{j+1}$ is
  augmenting.

  Let us now prove that $\{p_1, p_2, \dots\} \cap X$ induces a clique.
  We prove by induction on $j$ a stronger fact:
  $\{p_1, \dots, p_j\} \cap X \subseteq N^\uparrow[p_j] \cap X$ (which
  induces a clique by rule~\ref{a:l}).  This is clear for $j=1$, and
  assuming it is proved for a fixed $j$, it follows for $j+1$ from:
  $$\begin{array}{rl} 
      \{p_1, \dots, p_{j+1}\} \cap X  &= (\{p_1, \dots, p_j\} \cap X)
                                        \cup  (\{p_{j+1}\} \cap X)\\
                                      & \subseteq  (N^\uparrow[p_j] \cap X) \cup
                                        (\{p_{j+1}\} \cap X) \text{ (induction
                                        hypothesis)}\\
                                      & =
                                        (N^{\uparrow}(p_{j+1})\cap X) \cup
                                        (\{p_{j+1}\} \cap X) \text{
                                        ($p_jp_{j+1}$
                                        augmenting)}\\
                                      & =            N^{\uparrow}[p_{j+1}] \cap X                                                    
  \end{array}$$

  Hence, $\{p_1, p_2, \dots\} \cap X$ is a clique and
  therefore contains at most $k$ vertices.  Together with the vertices
  potentially in layers from 1 to $F(k+1) - 1$, we obtain that $V(P) \cap X$ contain at
  most $F(k+1) + k - 1$ vertices. 
\end{proof}

A separation $(A, B)$ of $G$ is \emph{fair} if there exists a pair
of vertical paths $P= p_i p_{i+1}\dots$, $Q = q_i q_{i+1}\dots$ such that:
\begin{itemize}
\item $p_i$ and $q_i$ are in the same layer $L_i$,
\item $|p_i \overrightarrow{L_i} q_i| \leq \ell - 1$,
\item The paths $P\setminus p_i = p_{i+1}p_{i+2}\dots $ and
  $Q\setminus q_i = q_{i+1}q_{i+2}\dots $ are augmenting paths,
\item $A = A(P, Q)$ and $B = B(P, Q)$ and
\item $|A \cap X| \geq n/3$ (recall that $n=|X|$).
\end{itemize}

Note that the notion of fair separation is defined for $G$ (and not
only for $G[X]$), but it depends on $X$.  Observe that $P$ and $Q$ are
possibly not augmenting (but removing their first vertex yields an
augmenting path).

\begin{lemma}
  \label{l:existF}
  There exists a fair separation in $G$. 
\end{lemma}

\begin{proof}
  Consider two distinct and non-adjacent vertices $p$ and $q$ in the
  first layer $L_1$.  Note that
  $|p \overrightarrow{L_1} q| \leq \ell - 1$ and
  $|q \overrightarrow{L_1} p| \leq \ell - 1$ because $L_1$ induces a
  cycle of length $\ell$ by the definition of layered wheels.  Let $P$
  and $Q$ be the two augmenting paths starting at $p$ and $q$
  respectively. By \cref{l:sep}, $(A(P, Q), B(P, Q))$ and
  $(A(Q, P), B(Q, P))$ are separations of $G$.  To prove that one of
  them is fair, only the condition on $|A\cap X|$ remains to be
  checked.  Since $P$ and $Q$ are vertex-disjoint by \cref{l:pathD},
  we have $A(P, Q) \cup A(Q, P) = V(G)$.  Hence, either
  $|A(P, Q) \cap X| \geq n/3$ or $|A(Q, P) \cap X| \geq n/3$, so the
  condition holds for at least one separation.
\end{proof}

Observe that in the following lemma, $F(k+1) = +\infty$ is possible.
In this case, the statement becomes trivial, but is stays true : it
says that some subset of $X$ (which is finite) has size at most
$+\infty$.

\begin{lemma}
  \label{l:balancedS}
  There exists a balanced separation of $G[X]$ of order at most
   $$ 2F(k+1) + (\ell +1) k -2.$$
\end{lemma}

\begin{proof}
  If $n \leq 5$, then the conclusion trivially holds because
  $\ell + 1 \geq 5$ and $F(k+1)\geq 2$, so $(X, X)$ is a separation
  satisfying the conclusion.  We therefore assume from here on that
  $n\geq 6$.

  Suppose now that $(A, B)$ is a fair separation with notation as in
  the definition.  The order of $(A \cap X, B \cap X)$ is at most
  $ 2F(k+1) + (\ell +1) k -2$ because by \cref{l:apb},
  $$A(P, Q) \cap B(P, Q) = V(P) \cup V(Q) \cup \bigcup_{u \in p_i
    \overrightarrow{L_i} q_i} N^\uparrow[u],$$ by \cref{l:sizeAug},
  $|V(P\setminus p_i) \cap X| \leq F(k+1) + k -1$, a similar
  inequality holds for $Q$,
  $|p_i \overrightarrow{L_i} q_i|\leq \ell-1$ since the separation is
  fair, and by rule~\ref{a:l} and the choice of $k$, for all
  $u\in p_i \overrightarrow{L_i} q_i$, $|N^\uparrow[u]| \leq k$.
  Also, by the definition of fair separations, $|A \cap X| \geq n/3$.
  Hence, $|(B\cap X)\setminus (A\cap X)|\leq 2n/3$.  Hence, either
  $(A \cap X, B \cap X)$ is balanced or
  $|(A\cap X) \setminus (B\cap X)| > 2n/3$.  We therefore assume from
  here on that for all fair separations $(A, B)$,
  $|(A\cap X) \setminus (B\cap X)| > 2n/3$ and we look for a
  contradiction.
  
  Note that for all fair separations $(A, B)$,
  $(A\cap X) \setminus (B\cap X)$ is non-empty.  By considering a
  layer $L_t$ containing a vertex of $X$ and maximal with this
  property (since $X$ is finite), we see that for all fair separations
  of $G$, the integer $i$ in the definition is at most $t-1$ (in fact,
  this implies that there exist finitely many fair separations).
  
  Since a fair separation exists by \cref{l:existF}, consider a fair
  separation $(A, B)$ of $G$ with notation as in the definition, such
  that $i$ is maximal (this is well defined since $i\leq t-1$) and,
  among all separations with $i$ maximal, such that
  $|p_{i+1} \overrightarrow{L_{i+1}} q_{i+1}|$ is minimal.

  Suppose that no internal vertex of
  $p_{i+1} \overrightarrow{L_{i+1}} q_{i+1}$ has a parent.  Then by
  rule~\ref{aN}, either $p_i=q_i$ or $p_iq_i\in A(G)$.  Moreover, by
  rules \ref{bS} and \ref{bB},
  $p_{i+1} \overrightarrow{L_{i+1}} q_{i+1}$ induces a path on
  $\ell-1$ vertices.  We then set $P' = P\setminus p_i$ and
  $Q'= Q \setminus q_i$, $A'= A(P', Q')$ and $B'=B(P', Q')$.  It is a
  routine matter to check that $(A', B')$ is a fair separation, except
  for the condition on $|A'\cap X|$.  But at most two vertices of $A'$
  are not in $A$, because by rule~\ref{bS} and~\ref{bB}, at most one
  vertex is in $N^\uparrow[p_i] \setminus N^\uparrow(p_{i+1})$ and at
  most one vertex is in
  $N^\uparrow[q_i] \setminus N^\uparrow(q_{i+1})$.  Hence,
  $|A'\cap X| \geq |A\cap X| - 2 > 2n/3 - 2 = n/3 + (n-6)/3 \geq n/3$
  since $n\geq 6$.  Consequently, the condition on the size of $A'\cap
  X$ is
  satisfied and $(A', B')$ contradicts the optimality of $(A, B)$
  since $i+1 > i$.

  We may therefore assume that some vertex $u$ in the interior of
  $p_{i+1} \overrightarrow{L_{i+1}} q_{i+1}$ has a parent~$v$.  By
  rule~\ref{aN}, we have $v\in p_{i} \overrightarrow{L_{i}} q_{i}$.
  Let $R'$ be the augmenting path out of $u$ and $R=vuR'$.  Set
  $A' = A(P, R)$, $A'' = A(R, Q)$, $B' = B(P, R)$ and $B'' = B(R, Q)$.
  It is a routine matter to check that $(A', B')$ and $(A'', B'')$ are
  fair separations, except for the condition on the size of $A'\cap X$ or
  $A''\cap X$.  We have $A = A' \cup A''$.  Hence, since $|A\cap X|\geq 2n/3$, either
  $|A'\cap X| \geq n/3$ or $|A''\cap X| \geq n/3$.  Hence, one of
  $(A', B')$ or $(A'', B'')$ is fair and contradicts the minimality of
  $|p_{i+1} \overrightarrow{L_{i+1}} q_{i+1}|$.
\end{proof}

The following is the main result about the treewidth of finite induced
subgraphs $H$ of layered wheels.  Observe that if
$F(\omega(H)+1) = +\infty$, then the first conclusion trivially holds.

\begin{lemma}
  \label{th:tw}
  For all integers $\ell \geq 4$ and all slow functions $f$, the
  $(f, \ell)$-layered wheel $G$ satisfies:
  \begin{itemize}
  \item For every finite induced subgraph $H$ of $G$:
    $$\tw(H) \leq 15 \left(2 F(\omega(H)+1)   + (\ell
      +1) \omega(H) -2\right).$$
  \item For all integers $k\geq 2$ such that $F(k-1)$ is finite and
    all integers $t\leq F(k)$, there exists a finite induced subgraph
    $H$ of $G$ satisfying:
    $$\omega(H) = k \text{ and } \tw(H)  \geq t-1.$$
  \end{itemize}
\end{lemma}

\begin{proof}
  Let us prove the first statement. Let $H'$ be an induced subgraph
  of~$H$.  Set $k= \omega(H')$.  By \cref{l:balancedS}, $H'$ has a
  balanced separation of order at most
  $2F(k+1) + (\ell +1) k -2 \leq 2F(\omega(H)+1) + (\ell +1) \omega(H)
  -2$.  Hence by \cref{th:DvNo},
  $$\tw(H)\leq  15 \left(2 F(\omega(H)+1) + (\ell +1) \omega(H) - 2\right).$$

  To prove the second statement, set $t' = \max(F(k-1) +1, t)$.
  Consider the graph $H$ induced by layers $L_1, \dots, L_{t'}$ of
  $G$.  Since $F(k-1)+1 \leq t' \leq F(k)$ (because $t\leq F(k)$),
  $f(t') = k$.  So by \cref{l:omega}, $\omega(H) = k$.  By
  \cref{l:layersCM}, $L_1, \dots, L_{t'}$ forms a clique minor of $H$.
  Hence, since $t'\geq t$, by \cref{l:Kminor}, $\tw(H) \geq t-1$.
\end{proof}

Observe that when $F(k)$ is finite and $F(k+1)$ is infinite,
\cref{th:tw} does not tell us whether the treewidth of induced
subgraphs of $G$ with clique number exactly $k$ is bounded or
not.  This is why we do not know the answer to \cref{q:hw}.

\section{Applications of layered wheels}
\label{sec:app}

Recall that when $f$ is a slow function, we set
$$F(k) = \sup\{i \in \mathbb N\setminus\{0\} \mid f(i) \leq k\}.$$  Call $F$ the
\emph{cumulative function of $f$}.  Recall that informally, $f(i)$
tells us what size of clique is obtained when adding the layer $L_i$.  This
number is 1 at the start, then 2, then 3, and then it grows by at
most 1 at each new layer.  Informally, $F(k)$ is the number of layers
where the clique number is at most~$k$.  Since $f$ is slow, we have
\[
\begin{array}{lc}
  F:\mathbb N\setminus\{0\} \rightarrow \mathbb N\setminus\{0\} \cup \{+\infty\},&\\
  F(1)=1, F(2)=2\text{ and}&\text{\rule{1cm}{0cm}}(\star)\\
  F(k+1) \geq F(k) +1\text{ for all }k\in \mathbb N\setminus\{0\}.&
 \end{array}
\]

It is clear that the $(f, \ell)$-layered wheel could be defined by giving $F$
instead of~$f$.  This one-to-one correspondance between $f$ and $F$
could be formalized by the fact that for all $i\in \mathbb N\setminus\{0\}$, we have
$$f(i) = \min \{k \in \mathbb N\setminus\{0\} \mid  F(k) \geq i\},$$  but we do not
need this.  We will just use freely the fact that any
slow function $f$ can be defined by describing its corresponding
cumulative function~$F$, provided that $F$
 satisfies the property $(\star)$.

\medskip

The following theorem answers \cref{Q:allFunctions}. 

\begin{theorem}
  \label{th:allFunctions}
  For every function $g: \mathbb N\setminus\{0\} \rightarrow \mathbb N\setminus\{0\}$ and
  every integer $\ell \geq 4$, there exists a $(\tw, \omega)$-bounded
  class of graphs $\cal C$ such that every hole in $\cal C$ has length
  at least $\ell$ and for all integers $k\geq 2$, there exists a graph
  $H\in \cal C$ satisfying
  $$\omega(H) = k \text{ and } \tw(H)  \geq g(k).$$
\end{theorem}

\begin{proof}
  Consider a graph $J$ of girth at least $\ell$ whose treewidth is at
  least $g(2)$ ($J$~can be a subdivision of a wall, or a (theta,
  triangle)-free layered wheel as defined
  in~\cite{DBLP:journals/jgt/SintiariT21}).  Set $F(1)=1$, $F(2)=2$
  and for all integers $k \geq 3$,
  $$F(k) = \max \{F(k-1) +1, g(k) +1\}.$$  Consider the slow function $f$
  whose cumulative function is~$F$.  It exists since $F(1)=1$,
  $F(2)=2$, and $F(k+1) \geq F(k) +1$ for all $k\in \mathbb
  N\setminus\{0\}$. Also, for all integers $k\geq 3$, $F(k)\geq g(k) +1$.

  Consider the class $\cal C$ of all finite induced subgraphs of
  either $J$ or the $(f, \ell)$-layered wheel $G$.  By
  \cref{l:lengthHoles}, every hole in $\cal C$ has length
  at least $\ell$.  By \cref{th:tw}, $\cal C$ is
  $(\tw, \omega)$-bounded.  For all integers $k\geq 3$, the second
  conclusion of \cref{th:tw} for $t=F(k)$ yields a graph $H\in \cal C$
  such that $\omega(H) = k$ and $\tw(H) \geq F(k) -1 \geq g(k)$.  For
  $k=2$, a graph $H\in \cal C$ such that $\omega(H) = k$ and
  $\tw(H) \geq g(k)$ also exists since $J\in \cal C$.
\end{proof}

The following theorem disproves \cref{conj:dallardMiSto}.  A function
$F:\mathbb N\setminus\{0\}\rightarrow \mathbb N\setminus\{0\}$ is
\emph{super-linear} if for every $c>0$, there exists an integer
$k\geq 1$ such that for all $x\geq k$, $F(x) > cx$.

\begin{theorem}
  \label{th:dallardMiSto}
  Let $\ell\geq 4$ be an integer and
  $F:\mathbb N\setminus\{0\}\rightarrow \mathbb N\setminus\{0\}$ be
  any super-linear function such that $F(1)=1$, $F(2)=2$, and
  $F(k+1) \geq F(k) +1$ for all $k\in \mathbb N\setminus\{0\}$.  Then
  there exists a hereditary class of graphs $\cal C$ such that every
  hole in $\cal C$ has length at least~$\ell$, $\cal C$ contains
  graphs of arbitrarily large tree-independence number and every
  $H \in \cal C$ satisfies
  $$\tw(H) \leq 15 \left(2F(\omega(H)+1)   + (\ell
    +1) \omega(H) -2\right).$$
\end{theorem}

\begin{proof}
  Consider the slow function $f$ whose cumulative function is $F$.
  Let~$G$ be the $(f, \ell)$-layered wheel and $\cal C$ be the class
  of finite induced subgraphs of~$G$.  By~\cref{l:lengthHoles}, the holes
  in $\cal C$ all have length at least $\ell$.  By \cref{th:tw}, the
  required bound on the treewidth holds.

  It remains to prove that $\cal C$ contains graphs of arbitrarily
  large tree-independence number.  So let $c\geq 1$ be an integer.
  Since $F$ is super-linear, let $k\geq 2$ be such that
  $F(k) \geq ck$.  Consider the graph $H$ induced by the layers
  $L_1, \dots, L_{F(k)}$ of $G$.  By \cref{l:omega}, $\omega(H) = k$
  and by \cref{l:layersCM}, the layers $L_1, \dots, L_{F(k)}$ form a
  clique minor of $H$.

  Consider any tree-decomposition
  $\mathcal T = (T, (X_s)_{s\in V (T )})$ of $H$.  By \cref{l:Kminor},
  there exists a node $s\in V(T)$ such that $X_s$ contains at least
  one vertex of each $L_i$, $i\in \{1, \dots, F(k)\}$.  Consider a
  subset $Y$ of $X_s$ that contains exactly one vertex in each layer
  $L_i$, $i\in \{1, \dots, F(k)\}$.  We have $|Y| = F[k]$ and
  $\omega(H[Y]) \leq k$.  By \cref{l:chordal}, $H[Y]$ is chordal.
  Hence by \cref{th:ch},
  $$\alpha(H[X_s]) \geq \alpha(H[Y]) \geq \frac{|Y|}{\chi(H[Y])} = \frac{F(k)}{\omega(H[Y])} \geq
  \frac{ck}{k} = c.$$

  Hence, for all tree-decompositions of $H$, some bag contains a
  stable of size at least $c$.  It follows that $\ta(H) \geq c$.  Since this
  can be performed for any integer $c$, $\cal C$ contains graph of arbitrarily
  large tree-independence number.
\end{proof}

By allowing $f$ to be eventually constant (or equivalently by allowing
infinite values of $F$), we obtain the following, that disproves 
\cref{conj:hajebiConj1} and hence also \cref{conj:hajebiConj2}.

\begin{theorem}
  \label{th:hajebiConj}
  For all integers $\ell \geq 5$, $c\geq 2$ and $t\geq 1$, there
  exists a graph $G$ of treewidth at least $t$ such that $\omega(G) \leq c+1$,
  every hole of $G$ has length at least $\ell$ (in particular, $G$
  contains no complete bipartite graph of treewidth at least~2) and
  every $K_c$-free induced subgraph of $G$ (in particular every
  $(c-2)$-degenerate induced subgraph of $G$) has treewidth at most
  $15 (2c + (\ell +1) (c-1) -2)$.
\end{theorem}

\begin{proof}
  Let $f$ be the function defined by $f(i)=\min \{i, c+1\}$.  Hence, $f$
  is slow and the cumulative function $F$ of $f$ satisfies $F(c) = c$
  and $F(c+1) = +\infty$.  Let $G$ be the graph induced by the layers
  $L_1$, \dots, $L_{t+1}$ of the $(f, \ell)$-layered wheel.  By
  \cref{l:omega}, $\omega(G) \leq  c+1$.  Since the layers $L_1$, \dots,
  $L_{t+1}$ form a clique minor of $G$ by \cref{l:layersCM}, $G$ has treewidth at
  least $t$ by \cref{l:Kminor}.

  Let $H$ be a $K_c$-free induced subgraph of $G$.  So
  $\omega(H) \leq c-1$.  By \cref{th:tw} and since $F(c)=c$,
  $\tw(H)\leq 15 (2c  + (\ell +1) (c-1) -2)$.  So $G$
  satisfies the conclusion. 
\end{proof}

\section*{Acknowledgement}

A first version of the $(f, \ell)$-layered wheel was presented at the
Graph Theory Workshop, organized by Sergey Norin, Paul Seymour and
David Wood at the Bellairs Research Institute, Barbados, in March
2024.  The discussions that arose from the discovery of a flaw in this
version were very helpful, in particular with Édouard Bonnet, Rose
McCarthy, Sergey Norin and Stéphan Thomassé.  We are grateful to
Sepehr Hajebi for useful discussions about bounding the treewidth by a
function of $\omega$ and Julien Duron for suggesting key ideas about
how to find a balanced separation.  Thanks to Hugo Jacob for pointing
out to us the construction
from~\cite{DBLP:journals/siamcomp/CorneilR05}. 

Maria Chudnovsky is supported by NSF-EPSRC Grant DMS-2120644, 
AFOSR grant FA9550-22-1-0083 and NSF Grant DMS-2348219.

Nicolas Trotignon is partially supported by the French National
Research Agency under research grant ANR DIGRAPHS ANR-19-CE48-0013-01
and the LABEX MILYON (ANR-10-LABX-0070) of Université de Lyon, within
the program Investissements d’Avenir (ANR-11-IDEX-0007) operated by
the French National Research Agency (ANR).

Part of this work was done when Nicolas Trotignon visited Maria
Chudnovsky at Princeton University with generous support of the
H2020-MSCA-RISE project CoSP- GA No. 823748.

\bibliographystyle{plain}
\bibliography{f-LayeredWheels}

\end{document}